\documentclass[reqno,a4paper, 11pt]{amsart}

\usepackage[a4paper=true,pdfpagelabels]{hyperref}
\usepackage{graphicx}

\usepackage[ansinew]{inputenc}
\usepackage{amsfonts,epsfig}
\usepackage{latexsym}
\usepackage{mathabx}
\usepackage{amsmath}
\usepackage{amssymb}

\usepackage{color}

\newtheorem{theorem}{Theorem}
\newtheorem{lemma}[theorem]{Lemma}

\newtheorem{conjecture}[theorem]{Conjecture}
\newtheorem{proposition}[theorem]{Proposition}

\theoremstyle{definition}
\newtheorem{definition}[theorem]{Definition}

\theoremstyle{remark}

\numberwithin{equation}{section}

\newcommand{\D}{\mathbb{D}}
\newcommand{\DD}{\widehat{\mathcal{D}}}
\newcommand{\Dd}{\widecheck{\mathcal{D}}}

\newcommand{\DDD}{\mathcal{D}}
\newcommand{\De}{\mathcal{D}}
\newcommand{\N}{\mathbb{N}}

\newcommand{\C}{\mathbb{C}}

\newcommand{\e}{\varepsilon}

\renewcommand{\phi}{\varphi}

\newcommand{\T}{\mathbb{T}}

       \def\b{\beta}        
\def\d{\delta}       \def\De{{\Delta}}    \def\e{\varepsilon}
     \def\om{\omega}      
              \def\f{\phi}
\def\p{\psi}         \def\r{\rho}         \def\z{\zeta}
                  \def\vp{\varphi}
\def\G{\Gamma}

\renewcommand{\H}{\mathcal{H}}

\begin{document}
\title[Compact differences of composition operators on Bergman spaces]{Compact differences of composition operators on Bergman spaces induced by doubling weights}


\keywords{Bergman space, Carleson measure, doubling weight, composition operator}

\author{Bin Liu}
\address{University of Eastern Finland, P.O.Box 111, 80101 Joensuu, Finland}
\email{binl@uef.fi}

\author{Jouni R\"atty\"a}
\address{University of Eastern Finland, P.O.Box 111, 80101 Joensuu, Finland}
\email{jouni.rattya@uef.fi}

{\author{Fanglei Wu}
\address{Department of Mathematics, Shantou University, Shantou, Guangdong 515063, China}
\address{University of Eastern Finland, P.O.Box 111, 80101 Joensuu, Finland}
\email{fangleiwu1992@gmail.com}}

\thanks{The first author is supported by China Scholarship Council No. 201706330108. The third author is supported by NNSF of China ( No. 11720101003).and NSF of Guangdong Province (No. 2018A030313512). This research were carried over when the third author visited University of Eastern Finland for 12 months. He wishes to express his gratitude to the stuff and the faculty of the Department of Physics and Mathematics for their hospitality during his visit.}

\begin{abstract}
Bounded and compact differences of two composition operators acting from the weighted Bergman space $A^p_\om$ to the Lebesgue space $L^q_\nu$, where $0<q<p<\infty$ and $\om$ belongs to the class~$\DDD$ of radial weights satisfying a two-sided doubling condition, are characterized. On the way to the proofs a new description of $q$-Carleson measures for $A^p_\om$, with $p>q$ and $\om\in\DDD$, involving pseudohyperbolic discs is established. This last-mentioned result generalizes the well-known characterization of $q$-Carleson measures for the classical weighted Bergman space $A^p_\alpha$ with $-1<\alpha<\infty$ to the setting of doubling weights. The case $\om\in\DD$ is also briefly discussed and an open problem concerning this case is posed.
\end{abstract}

\maketitle

\section{Introduction and main results}

Each analytic self-map $\vp$ of the unit disc $\D=\{z\in\C:|z|<1\}$ induces the composition operator $C_\varphi$, defined by $C_\varphi f=f\circ\varphi$, acting on the space $\H(\D)$ of analytic functions in~$\D$. These operators have been extensively studied in a variety of function spaces, see for example~\cite{CCM,SJH1,SJH2,Z}. Some studies on topological properties of the space of composition operators attracted attention towards differences of composition operators. The question of when the difference $C_\f-C_\p$ of two composition operators is compact on the Hardy space~$H^p$ was posed by Shapiro and Sundberg~\cite{SS}. Differences of composition operators have been studied ever since by many authors on several function spaces, see, for example,~\cite{AW,CKP,GTE,LR2020}. In~2004, Nieminen and Saksman ~\cite{NS} showed that the compactness of $C_\f-C_\p$ on the Hardy space~$H^p$, with $1\leq p<\infty$, is independent of $p$. By using this, Choe, Choi, Koo and Yang~\cite{CB} then characterized compact operators $C_\f-C_\p$ on Hardy spaces by using Carleson measures in~2020. Further, Moorhouse~\cite{MJ,MT} characterized the compactness of $C_\f-C_\p$ on the standard weighted Bergman space $A^2_{\alpha}$. Saukko~\cite{SE,SEO} generalized Moorhouse's results by characterizing the compactness from $A^p_\alpha$ to $A^q_\beta$ if either $1<p\le q$, or $p>q\ge1$. Very recently, Shi, Li and Du ~\cite{SLD} extended Saukko's results for the complete range $0<p,q<\infty$ and to higher dimensions.

In this paper we are interested in the compactness of $C_\f-C_\p$ on weighted Bergman spaces induced by doubling weights. We proceed with necessary definitions. For $0<p<\infty$ and a positive Borel measure $\nu$ on $\D$, the Lebesgue space~$L^p_\nu$ consists of complex valued $\nu$-measurable functions $f$ on $\D$ such that
    $$
    \|f\|_{L^p_\nu}^p=\int_\D|f(z)|^p\,d\nu(z)<\infty.
    $$
If $\nu$ is continuous with respect to the Lebesgue measure, that is, $d\nu(z)=\om(z)dA(z)$ for some non-negative function $\om$, then we adopt the notation $L^p_\nu=L^p_\om$ without arising any confusion. For a nonnegative function $\om\in L^1([0,1))$, its extension to $\D$, defined by $\om(z)=\om(|z|)$ for all $z\in\D$, is called a radial weight. For $0<p<\infty$ and a radial weight $\omega$, the weighted Bergman space~$A^p_\omega$ consists of analytic functions in $L^p_\om$.  As usual,~$A^p_\alpha$ stands for the classical weighted Bergman space induced by the standard radial weight $\omega(z)=(1-|z|^2)^\alpha$, where $-1<\alpha<\infty$. For the theory of these spaces consult \cite{DS,HKZ}.

For a radial weight $\om$, write $\widehat{\om}(z)=\int_{|z|}^1\om(s)\,ds$ for all $z\in\D$. In this paper we always assume $\widehat{\om}(z)>0$, for otherwise $A^p_\om=\H(\D)$ for each $0<p<\infty$. A radial weight $\om$ belongs to the class~$\DD$ if
there exists a constant $C=C(\om)\ge1$ such that $\widehat{\om}(r)\le C\widehat{\om}(\frac{1+r}{2})$ for all $0\le r<1$.
Moreover, if there exist $K=K(\om)>1$ and $C=C(\om)>1$ such that $\widehat{\om}(r)\ge C\widehat{\om}\left(1-\frac{1-r}{K}\right)$ for all $0\le r<1$, then we write $\om\in\Dd$. In other words, $\om\in\Dd$ if there exists $K=K(\om)>1$ and $C'=C'(\om)>0$ such that
	$$
	\widehat{\om}(r)\le C'\int_r^{1-\frac{1-r}{K}}\om(t)\,dt,\quad 0\le r<1.
	$$
The intersection $\DD\cap\Dd$ is denoted by $\DDD$, and this is the class of weights that we mainly work with. For recent developents on Bergman spaces induced by weights in $\DD$, see \cite{PelRat2020} and the reference therein.

In this paper we consider compact differences of two composition operators from the weighted Bergman space $A^p_\omega$ to the Lebesgue space $L^q_\nu$ when $0<p,q<\infty$ and $\omega\in\DD$. To state the first main result, write
		\begin{equation*}
    \delta(z)=\frac{\psi(z)-\varphi(z)}{1-\overline{\psi(z)}\varphi(z)},\quad z\in\D.
    \end{equation*}
The next result generalizes \cite[Theorem 1.2]{SEO} to the setting of doubling weights.

\begin{theorem}\label{Theorem:main-1}
Let $0<q<p<\infty$ and $\om\in\DDD$, and let $\nu$ be a positive Borel measure on $\D$. Let $\f$ and $\p$ be analytic self-maps of $\D$. Then the following statements are equivalent:
\begin{itemize}
\item[\rm(i)] $C_\f-C_\p:A^p_\om\rightarrow L^q_\nu$ is bounded;
\item[\rm(ii)] $C_\f-C_\p:A^p_\om\rightarrow L^q_\nu$ is compact;
\item[\rm(iii)] $\d C_\f$ and $\d C_\p$ are compact (or equivalently bounded) from $A_\om^p$ to $L^q_\nu$.
\end{itemize}
\end{theorem}

The proof of Theorem~\ref{Theorem:main-1} is organized as follows. We first show that $C_\f-C_\p$ is compact if $\d C_\f$ and $\d C_\p$ are bounded. The proof of this implication is straightforward and relies on the fact that the norm of $f\in\H(\D)$ in $A^p_\om$ is comparable to the $L^p_\om$-norm of the non-tangential maximal function $N(f)(z)=\sup_{\z\in\Gamma(z)}|f(\z)|$, where 
	$$
	\Gamma(z)=\left\{\z\in\D:|\theta-\arg \z|<\frac{1}{2}\left(1-\frac{|\z|}{r}\right)\right\},\quad z=re^{i\theta}\in\overline{\D}\setminus\{0\},
	$$ 
is a non-tangential approach region with vertex at $z$. Then, as each compact operator is bounded, the proof boils down to showing that $\d C_\f$ and $\d C_\p$ are compact whenever $C_\f-C_\p$ is bounded. This part of the proof is more laborious. We first observe that for each $\r$-lattice $\{z_k\}$ the function
	$$
	z\mapsto\sum_kb_k\left(\frac{1-|z_k|}{1-\overline{z_k}z}\right)^M\frac{1}{(\omega(T(z_k)))^{\frac{1}{p}}}
	$$
belongs to $A^p_{\omega}$ for all $b=\{b_k\}\in\ell^p$ and its $A^p_\om$-norm is dominated by a universal constant times~$\|b\|_{\ell^p}$. Then we use this function for testing and apply Khinchine's inequality. After duality arguments we finally arrive to a situation where we must understand well the continuous embeddings $A^p_\om\subset L^q_\mu$. Recall that a positive Borel measure $\mu$ on $\D$ is a $q$-Carleson measure for $A^p_\om$ if the identity operator $I:A^p_\om\to L^q_\mu$ is bounded. A complete characterization of such measures in the case $\om\in\DD$ can be found in \cite{PelRatEmb}, see also \cite{PelSum14,PR,PRS}. In particular, it is known that if $p>q$ and $\om\in\DD$, then $\mu$ is a $q$-Carleson measure for $A^p_\om$ if and only if the function
	\begin{equation}\label{kalle}
	\z\mapsto\int_{\Gamma(\z)}\frac{d\mu(z)}{\om(T(z))}
	\end{equation}
belongs to $L^\frac{p}{p-q}_\om$. Here and from now on $T(z)=\{\z\in\D:z\in\Gamma(\z)\}$ is the tent induced by $z\in\D\setminus\{0\}$. Further, $\om(E)=\int_E\om dA$ for each measurable set $E\subset\D$. Observe that in \eqref{kalle} we may replace the tent $T(z)$ by the Carleson square $S(z)=\{\z:1-|z|<|\z|<1,\,|\arg\z-\arg z|<(1-|z|)/2\}$ because $\om(T(z))\asymp\om(S(z))$ for all $z\in\D\setminus\{0\}$ if $\om\in\DD$. To complete the proof of Theorem~\ref{Theorem:main-1} we will need a variant of the above characterization of Carleson measures in the case $\om\in\DDD$. Our statement involves pseudohyperbolic discs and an auxiliary weight associated to $\om$. The pseudohyperbolic distance between two points $a$ and $b$ in $\D$ is $\rho(a,b)=|(a-b)/(1-\overline{a}b)|$. For $a\in\D$ and $0<r<1$, the pseudohyperbolic disc of center $a$ and of radius $r$ is $\Delta(a,r)=\{z\in \D:\rho(a,z)<r\}$. It is well known that $\Delta(a,r)$ is an Euclidean disk centered at $(1-r^2)a/(1-r^2|a|^2)$ and of radius $(1-|a|^2)r/(1-r^2|a|^2)$. We denote $\widetilde{\omega}(z)=\widehat{\omega}(z)/(1-|z|)$ for all $z\in\D$ and note that
	\begin{equation}\label{eq:equivalent norms}
	\|f\|_{A^p_{\widetilde{\om}}}\asymp\|f\|_{A^p_\om},\quad f\in\H(\D),
	\end{equation}
provided $\om\in\DDD$, by \cite{PelRat2020}. Our embedding theorem generalizes the case $n=0$ of \cite[Theorem~1]{L1993} to doubling weights and reads as follows.

\begin{theorem}\label{Theorem:cm}
Let $0<q<p<\infty$ and $\omega\in\mathcal{D}$, and let $\mu$ be a positive Borel measure on $\D$. Then the following statements are equivalent:
\begin{itemize}
\item[\rm(i)] $\mu$ is a $q$-Carleson measure for $A^p_{\omega}$;
\item[\rm(ii)] $I:A^p_\om\to L^q_\mu$ is compact;
\item[\rm(iii)] the function
	$$
	\Theta^\om_\mu(z)=\frac{\mu(\Delta(z,r))}{\omega(S(z))},\quad z\in\D\setminus\{0\},
	$$
belongs to $L_{\widetilde{\omega}}^{\frac{p}{p-q}}$ for some (equivalently for all) $r\in(0,1)$.
\end{itemize}
Moreover,
	\begin{equation}\label{operator norm}
	\|I\|_{A^p_\om\to L^q_\mu}^q\asymp\|\Theta^\om_\mu\|_{L_{\widetilde{\omega}}^{\frac{p}{p-q}}}.
	\end{equation}
\end{theorem}

We may not replace $L_{\widetilde{\omega}}^{\frac{p}{p-q}}$ by $L_{{\omega}}^{\frac{p}{p-q}}$ in part (iii) of Theorem~\ref{Theorem:cm}. A counter example can be constructed as follows. Write $D(z,r)$ for the Euclidean disc $\{\z:|\z-z|<r\}$. Let $r_n=1-2^{-n}$ and $A_n=D(0,r_{n+1})\setminus D(0,r_n)$ for all $n\in\N$. Pick up an $\om\in\DDD$ such that it vanishes on $A_{2n}$ for all $n\in\N$. A simple example of a such weight is $\sum_{n\in\N}\chi_{A_{2n+1}}$.
Then choose $\mu$ such that for some $\e>0$ its support is contained in the union of the discs $\Delta(a_n,\e)$ which have the property that for some fixed $r\in(0,1)$ we have $\Delta(z,r)\subset A_{2n}$ for all $z\in \Delta(a_n,\e)$ and for all $n\in\N$. The choice $a_n=(r_{2n}+r_{2n+1})/2$ works if $0<r<1$ and $\e=\e(r)>0$ are sufficiently small. Then, for such an $r$, the norm $\|\Theta^\om_\mu\|_{L_{\omega}^{\frac{p}{p-q}}}$ vanishes and thus it cannot be comparable to $\|\Theta^\om_\mu\|_{L_{\widetilde{\omega}}^{\frac{p}{p-q}}}$ which is non-zero if $\mu$ is not a zero measure because $\widetilde{\om}$ is strictly positive. Moreover, by choosing $\mu$ appropriately the norm $\|\Theta^\om_\mu\|_{L_{\widetilde{\omega}}^{\frac{p}{p-q}}}$ can be made infinite.

The second main result of this study concerns the case when $p<q$ and $\om\in\DD$. It completes in part the main result in \cite{LR2020} which concerns the class $\DDD$ only. An analogue of this result was proved for the Hardy spaces in~\cite{SL}. Therefore Theorem~\ref{Theorem:main-2} takes care of the gap consisting of small Bergman spaces that exists between the Hardy and the standard weighted Bergman spaces.

\begin{theorem}\label{Theorem:main-2}
Let $0<p<q<\infty$ and $\om\in\DD$, and let $\nu$ be a positive Borel measure on $\D$. Let $\f$ and $\p$ be analytic self-maps of $\D$. Then $C_\f-C_\p:A^p_\om\rightarrow L^q_\nu$ is bounded (resp. compact) if and only if $\d C_\f$ and $\d C_\p$ are bounded (resp. compact) from $A_\om^p$ to $L^q_\nu$.
\end{theorem}

If $q=p$ then the boundedness (resp. compactness) of $\d C_\f$ and $\d C_\p$ implies the same property for $C_\f-C_\p$ by Proposition~\ref{proposition pq11} below. Further, Proposition~\ref{proposition pq1} below shows that $C_\f-C_\p$ is compact if $\d C_\f$ and $\d C_\p$ are bounded when $p>q$. But we do not know if the boundedness of $C_\f-C_\p$ implies that of $\d C_\f$ and $\d C_\p$ if $\om\in\DD\setminus\DDD$ if $p\ge q$. The methods used in this paper do not seem to give this implication and hence this case remains unsettled.

The rest of the paper is organized as follows. In the next section we consider Carleson embeddings and prove Theorem~\ref{Theorem:cm}. The sufficiency of the conditions on $\d C_\f$ and $\d C_\p$ are established in Section~\ref{Sec:sufficiency}, while Section~\ref{Sec:necessity} is devoted to their necessity. Finally, in Section~\ref{final} we indicate how our main findings on $C_\f-C_\p$ follow from these results.

To this end, couple of words about the notation used in this paper. The letter $C=C(\cdot)$ will denote an absolute constant whose value depends on the parameters indicated
in the parenthesis, and may change from one occurrence to another.
We will use the notation $a\lesssim b$ if there exists a constant
$C=C(\cdot)>0$ such that $a\le Cb$, and $a\gtrsim b$ is understood
in an analogous manner. In particular, if $a\lesssim b$ and
$a\gtrsim b$, then we write $a\asymp b$ and say that $a$ and $b$ are comparable.

\section{Carleson measures}

If $\om\in\DDD$, then there exist constants $0<\alpha=\alpha(\om)\le\b=\b(\om)<\infty$ and $C=C(\om)\ge1$ such that
	\begin{equation}\label{Eq:characterization-D}
	\frac1C\left(\frac{1-r}{1-t}\right)^\alpha
	\le\frac{\widehat{\om}(r)}{\widehat{\om}(t)}
	\le C\left(\frac{1-r}{1-t}\right)^\beta,\quad 0\le r\le t<1.
	\end{equation}
In fact, these inequalities characterize the class $\DDD$ because the right hand inequality is satisfied if and only if $\om\in\DD$ by \cite[Lemma~2.1]{PelSum14}, while the left hand inequality describes the class $\Dd$ in an analogous manner \cite[(2.27)]{PelRat2020}. The inequalities \eqref{Eq:characterization-D} will be frequently used throughout the paper.

\medskip
\noindent\emph{Proof of} Theorem~\ref{Theorem:cm}.
If $\mu$ is a $q$-Carleson measure for $A^p_{\omega}$, then $I:A^p_\om\to L^q_\mu$ is automatically compact by \cite[Theorem~3(iii)]{PRS}. Therefore it suffices to show that (i) and (iii) are equivalent and establish \eqref{operator norm}.

Assume first $\Theta^\om_\mu\in L_{\widetilde{\omega}}^{\frac{p}{p-q}}$ for some $r\in(0,1)$. The subharmonicity of $|f|^q$, Fubini's theorem, H\"{o}lder's inequality and \eqref{eq:equivalent norms} imply
	\begin{equation*}
	\begin{split}
	\|f\|^q_{L_{\mu}^q}
	&\lesssim\int_{\D}\left(\int_{\Delta(z,r)}\frac{|f(\zeta)|^q}{(1-|\zeta|^2)^2}dA(\zeta)\right)d\mu(z)\\
	&=\int_{\D}\left(\int_{\Delta(z,r)}|f(\zeta)|^q\frac{\widetilde{\omega}(\zeta)}{\widehat{\omega}(\z)(1-|\z|)}dA(\zeta)\right)d\mu(z)\\
	&\le\int_{\D}|f(\zeta)|^q\frac{\mu(\Delta(\zeta,r))}{\omega(S(\zeta))}\widetilde{\omega}(\zeta)dA(\zeta)
	\le\|f\|^{q}_{A^p_{\widetilde{\omega}}}
	\left\|\Theta^\om_\mu\right\|_{L^{\frac{p}{p-q}}_{\widetilde{\omega}}}\\
	&\asymp\|f\|^{q}_{A^p_{\omega}}
	\left\|\Theta^\om_\mu\right\|_{L^{\frac{p}{p-q}}_{\widetilde{\omega}}},\quad f\in A^p_\om.
	\end{split}
	\end{equation*}
Therefore $\mu$ is a $q$-Carleson measure for $A^p_\om$ and $\|I\|_{A^p_\om\to L^q_\mu}^q\lesssim\left\|\Theta^\om_\mu\right\|_{L^{\frac{p}{p-q}}_{\widetilde{\omega}}}$.

Conversely, assume that $\mu$ is a $q$-Carleson measure for $A^p_\om$. Then \eqref{eq:equivalent norms} shows that $\mu$ is also a $q$-Carleson measure for $A^p_{\widetilde\om}$ and the corresponding operator norms are comparable. Further, since $\om\in\DDD$ by the hypothesis, an application of \eqref{Eq:characterization-D} shows that $\widetilde{\om}\in\DDD$. Therefore \cite[Theorem~1(a)]{PelRatEmb} implies
	$$
	\|B^{\widetilde{\om}}_\mu\|_{L^{\frac{p}{p-q}}_{\widetilde{\om}}}^{\frac{p}{p-q}}
	=\int_\D\left(\int_{\Gamma(z)}\frac{d\mu(\z)}{\widetilde{\om}(T(\zeta))}\right)^\frac{p}{p-q}\widetilde{\om}(z)\,dA(z)<\infty,
	$$
where
	$$
	\Gamma(z)=\left\{\zeta\in\D:|\arg\z-\arg z|<\frac12\left(1-\frac{|\zeta|}{|z|}\right)\right\},\quad z\in\D\setminus\{0\},
	$$
is a non-tangential approach region with vertex at $z$. Further, by \cite[Theorem~3(iii)]{PRS} we have $\|I\|_{A^p_\om\to L^q_\mu}^q\gtrsim\left\|B^{\widetilde{\om}}_\mu\right\|_{L^{\frac{p}{p-q}}_{\widetilde{\omega}}}$.

Let now $r\in(0,1)$ be given. For $K>1$ and $z\in\D\setminus D(0,1-\frac1K)$ write $z_K=(1-K(1-|z|))e^{i\arg z}$. Pick up $K=K(r)>1$ and $R=R(r)\in(1-\frac1K,1)$ sufficiently large such that $\Delta(z_K,r)\subset\Gamma(z)$ for all $z\in\D\setminus D(0,R)$. Straightforward applications of the left hand inequality in \eqref{Eq:characterization-D} show that $\widetilde{\om}(T(\z))\lesssim\om(S(\z))$, as $|\z|\to1^-$, and $\widetilde{\om}(z)\lesssim\widetilde{\om}(z_K)$ for all $z\in\D\setminus D(0,R)$. In an analogous way we deduce $\om(S(\z))\lesssim\om(S(z_K))$ for all $\z\in\Delta(z_K,r)$ and $z\in\D\setminus D(0,R)$ by using the right hand inequality. Therefore
	\begin{equation*}
	\begin{split}
	\infty
	&>\|I\|_{A^p_\om\to L^q_\mu}^q\gtrsim\int_{\D\setminus D(0,R)}\left(\int_{\Delta(z_K,r)}\frac{d\mu(\z)}{\om(T(\zeta))}\right)^\frac{p}{p-q}\widetilde{\om}(z)\,dA(z)\\
	&\gtrsim\int_{\D\setminus D(0,R)}\left(\frac{\mu(\Delta(z_K,r))}{\om(S(z_K))}\right)^\frac{p}{p-q}\widetilde{\om}(z_K)\,dA(z)\\
	&\asymp\int_{\D\setminus D(0,1-K(1-R))}\left|\Theta^\om_\mu(z)\right|^\frac{p}{p-q}\widetilde{\om}(z)\,dA(z).
	\end{split}
	\end{equation*}
It follows that $\Theta^\om_\mu\in L_{\widetilde{\omega}}^{\frac{p}{p-q}}$ and $\|I\|_{A^p_\om\to L^q_\mu}^q\gtrsim\left\|B^{\widetilde{\om}}_\mu\right\|_{L^{\frac{p}{p-q}}_{\widetilde{\omega}}}\gtrsim\|\Theta^\om_\mu\|_{L_{\widetilde{\omega}}^{\frac{p}{p-q}}}$.
\hfill$\Box$

\section{Sufficient conditions}\label{Sec:sufficiency}

In this section we establish sufficient conditions for $C_\f-C_\p:A^p_\om\rightarrow L^q_\nu$ to be bounded or compact. All these results are valid under the hypothesis $\om\in\DD$ despite the main results stated in the introduction concern only the class $\DDD$. We begin with the case $p\le q$.

\begin{proposition}\label{proposition pq11}
Let $0<p\le q<\infty$ and $\om\in\DD$, and let $\nu$ be a positive Borel measure on $\D$. Let $\f$ and $\p$ be analytic self-maps of $\D$. If $\d C_\f$ and $\d C_\p$ are bounded (resp. compact) from $A_\om^p$ to $L^q_\nu$, then $C_\f-C_\p:A^p_\om\rightarrow L^q_\nu$ is bounded (resp. compact).
\end{proposition}

\begin{proof}
We begin with the statement on the boundedness. Let first $q>p$. Let $f\in A^p_\omega$ with $\|f\|_{A^p_\om}\le1$. Fix $0<r<R<1$, and denote $E=\{z\in \D: |\delta(z)|<r\}$ and $E'=\D\setminus E$. Write
      \begin{equation*}
      (C_\varphi-C_\psi)(f)=(C_\varphi-C_\psi)(f)\chi_{E'}+(C_\varphi-C_\psi)(f)\chi_{E},
      \end{equation*}
and observe that it is enough to prove that the quantities
      \begin{equation}\label{Eq:3-conditions}
      \|(C_\varphi-C_\psi)(f)\chi_{E'}\|_{L^q_\nu}\quad\textrm{and}\quad \|(C_\varphi-C_\psi)(f)\chi_{E}\|_{L^q_\nu}
      \end{equation}
are bounded.

We begin with considering the first quantity in \eqref{Eq:3-conditions}. By the definition of the set $E$ we have the estimate
    \begin{equation}\label{Eq:easy}
    |(C_\varphi-C_\psi)(f)\chi_{E'}|\le\frac{1}{r}(|\d C_\varphi(f)|+|\d C_\psi(f)|),
	  \end{equation}
on $\D$. Since the operators $\d C_\f$ and $\d C_\p$ both are bounded from $A_\om^p$ to $L^q_\nu$ by the hypothesis, the first term in \eqref{Eq:3-conditions} is bounded by \eqref{Eq:easy}.

We next show that also the second term in \eqref{Eq:3-conditions} is bounded. Let $\mu$ be a finite nonnegative Borel measure on $\D$ and $h$ a measureable function on $\D$. For an analytic self-map $\varphi$ of $\D$, the weighted pushforward measure is defined by
       \begin{equation}\label{Eq:pusforward}
       \varphi_*(h,\mu)(M)=\int_{\varphi^{-1}(M)}h d\mu
       \end{equation}
for each measurable set $M\subset\D$. If $\mu$ is the Lebesgue measure, we omit the measure in the notation and write $\varphi_*(h)(M)$ for the left hand side of \eqref{Eq:pusforward}. By the measure theoretic change of variable~\cite[Section~39]{PM}, we have $\|\delta C_\varphi(f)\|_{L^q_\nu}=\|f\|_{L^q_{\varphi_*(|\delta|^q\nu)}}$ for each $f\in A^p_\om$. Therefore the identity operator from $A^p_\om$ to $L^q_{\varphi_*(|\delta|^q\nu)}$ is bounded by the hypothesis. Hence $\varphi_*(|\delta |^q\nu)(\Delta(\z,R))\lesssim\omega(S(\z))^{\frac{q}{p}}$ for all $\z\in\D\setminus\{0\}$ by \cite[Theorem~1(c)]{PelRatEmb}. Further, by \cite[Lemma~3]{LR2020}, with $\om\equiv1$ and $q=p$, there exists a constant
  $C=C(p,r,R)>0$ such that
     \begin{equation*}
     |f(z)-f(a)|^q\le C\frac{\rho(z,a)^p}{(1-|a|)^2}\int_{\Delta(a,R)}|f(\zeta)|^p\,dA(\zeta), \quad a\in \D,\quad z\in\Delta(a,r),
     \end{equation*}
for all $f\in A^p_\om$ with $\|f\|_{A^p_\om}\le1$. This and Fubini's theorem yield
		\begin{equation}\label{Eq:estimate-sufficiency}
		\begin{split}
    \|(C_\varphi-C_\psi)(f)\chi_E\|^q_{L^q_\nu}
		&=\int_E|f(\varphi(z))-f(\psi(z))|^q\nu(z)\,dA(z)\\
    &\lesssim\int_E\frac{|\delta(z)|^q}{(1-|\varphi(z)|)^2} \int_{(\varphi(z),R)}|f(\z)|^q\,dA(\z)\nu(z)\,dA(z)\\
    &\le\int_{\D}|f(\z)|^q\left(\int_{\varphi^{-1}(\Delta(\z,R))\cap E}\frac{|\delta(z)|^q}{(1-|\varphi(z)|)^2}\nu(z)\,dA(z)\right)dA(\z)\\
    &\lesssim\int_\D|f(\z)|^q\left(\int_{\varphi^{-1}(\Delta(\z,R))}\frac{|\delta(z)|^q}{(1-|\z|)^2}\nu(z)\,dA(z)\right)dA(\z)\\
    &=\int_\D|f(\z)|^q\frac{\varphi_*(|\delta |^q\nu)(\Delta(\z,R))}{(1-|\z|)^2}\,dA(\z)\\
		&\lesssim\int_\D|f(\z)|^q\frac{\omega(S(\z))^{\frac{q}{p}}}{(1-|\z|)^2}\,dA(\z)
		=\int_\D|f(\z)|^q\,d\mu(\z).
    \end{split}
    \end{equation}
Standard arguments show that $\mu(S(a))\lesssim \omega(S(a))^{\frac{q}{p}}$ for all $a\in\D\setminus\{0\}$. Hence \cite[Theorem~1(c)]{PelRatEmb} yields $\|(C_\varphi-C_\psi)(f)\chi_E\|^q_{L^q_\nu}\lesssim \|f\|^{q}_{L^q_\mu}\lesssim\|f\|^q_{A^p_\omega}$. Therefore also the second term in \eqref{Eq:3-conditions} is bounded. This finishes the proof of the case $q>p$.

Let now $q=p$. By following the proof above, it suffices to show that
	\begin{equation}\label{eq:enough}
	\int_{S}\frac{\varphi_*(|\delta |^q\nu)(\Delta(\z,R))}{(1-|\z|)^2}\,dA(\z)\lesssim\om(S)
	\end{equation}
for every Carleson square $S\subset\D$. By the hypothesis, the identity operator from $A^p_\om$ to $L^p_{\varphi_*(|\delta|^q\nu)}$ is bounded, and hence $\varphi_*(|\delta |^q\nu)(S)\lesssim\omega(S)$ for all $S$ by \cite[Theorem~1(b)]{PelRatEmb}. But for each positive Borel measure $\mu$ on $\D$, Fubini's theorem yields
	\begin{equation}\label{pili}
	\begin{split}
	\int_{S(a)}\frac{\mu(\Delta(\z,R))}{(1-|\z|)^2}\,dA(\zeta)
	&=\int_{\{z\in\D:S(a)\cap\Delta(z,R)\ne\emptyset\}}
	\left(\int_{S(a)\cap\Delta(z,R)}\frac{dA(\z)}{(1-|\z|)^2}\right)\,d\mu(z)\\
	&\le\int_{S(b)}\left(\int_{\Delta(z,R)}\frac{dA(\z)}{(1-|\z|)^2}\right)\,d\mu(z)
	\asymp\mu(S(b)),\quad|a|>R,
	\end{split}
	\end{equation}
where $b=b(a,R)\in\D$ satisfies $\arg b=\arg a$ and $1-|b|\asymp1-|a|$ for all $a\in\D\setminus\overline{D(0,R)}$. By applying this to $\mu=\varphi_*(|\delta|^q\nu)$ and using the hypothesis $\om\in\DD$ we deduce
	$$
	\int_{S(a)}\frac{\varphi_*(|\delta |^q\nu)(\Delta(\z,R))}{(1-|\z|)^2}\,dA(\z)
	\lesssim\varphi_*(|\delta|^q\nu)(S(b))\lesssim\om(S(b))\lesssim\om(S(a)),\quad |a|>R.
	$$
This estimate implies \eqref{eq:enough}, and thus the case $q=p$ is proved.

To obtain the compactness statement, it suffices to show that the quantities
	\begin{equation}\label{Eq:4-conditions-comp}
  \|(C_\varphi-C_\psi)(f_n)\chi_{E'}\|_{L^q_\nu},\quad\textrm{and}\quad
	\|(C_\varphi-C_\psi)(f_n)\chi_{E}\|_{L^q_\nu}
  \end{equation}
tend to zero as $n\rightarrow\infty$ for each sequence $\{f_n\}_{n\in\N}$ in $A^p_\om$ which tends to zero uniformly on compact subsets of $\D$ as $n\rightarrow\infty$ and satisfies $\|f_n\|_{A^p_\om}\le1$ for all $n\in\N$. Since $\d C_\f$ and $\d C_\p$ are compact from $A_\om^p$ to $L^q_\nu$ by the hypothesis, an application of \eqref{Eq:easy} to $f=f_n$ shows that the first quantity in \eqref{Eq:4-conditions-comp} tends to zero as $n\to\infty$. As for the second quantity, observe that \eqref{Eq:estimate-sufficiency} implies
	\begin{equation}\label{Eq:estimate-sufficiency-compactness}
		\begin{split}
    \|(C_\varphi-C_\psi)(f_n)\chi_E\|^q_{L^q_\nu}
		&\lesssim\int_\D|f_n(\z)|^q\frac{\varphi_*(|\delta |^q\nu)(\Delta(\z,R))}{(1-|\z|)^2}\,dA(\z),\quad n\in\N.
		\end{split}
    \end{equation}
Let first $q>p$. Since the identity operator from $A^p_\om$ to $L^q_{\varphi_*(|\delta|^q\nu)}$ is compact by the hypothesis, we have $\varphi_*(|\delta |^q\nu)(S(\z))/\omega(S(\z))^{\frac{q}{p}}\to0$ as $|\z|\to1^-$ by \cite[Theorem~3(ii)]{PRS}. Now, for each $\z\in\D\setminus\{0\}$ pick up $\z'=\z'(\z,R)\in\D$ such that $\arg\z'=\arg\z$, $\Delta(\z',R)\subset S(\z)$ and $1-|\z'|\asymp1-|\z|$ for all $\z\in\D\setminus\{0\}$. Then
	$$
	\frac{\varphi_*(|\delta |^q\nu)(S(\z))}{\omega(S(\z))^{\frac{q}{p}}}
	\ge\frac{\varphi_*(|\delta |^q\nu)(\Delta(\z',R))}{\omega(S(\z))^{\frac{q}{p}}}
	\gtrsim\frac{\varphi_*(|\delta |^q\nu)(\Delta(\z',R))}{\omega(S(\z'))^{\frac{q}{p}}},\quad \z\in\D\setminus\{0\},
	$$
and hence $\sup_{\z\in\D\setminus\{0\}}\varphi_*(|\delta |^q\nu)(\Delta(\z,R))/\omega(S(\z))^{\frac{q}{p}}<\infty$ and, for a given $\e>0$, there exists $\eta=\eta(\e)\in(0,1)$ such that $\varphi_*(|\delta |^q\nu)(\Delta(\z,R))/\omega(S(\z))^{\frac{q}{p}}<\e$ for all $\z\in\D\setminus D(0,\eta)$. Further, by the uniform convergence, there exists $N=N(\e,\eta,q)\in\N$ such that $|f_n|^q\le\e$ on $D(0,\eta)$ for all $n\ge N$. These observations together with the proof of the boundedness case and \eqref{Eq:estimate-sufficiency-compactness} yield
	\begin{equation*}
	\begin{split}
  \|(C_\varphi-C_\psi)(f_n)\chi_E\|^q_{L^q_\nu}
	&\lesssim\sup_{\zeta\in D(0,\eta)\setminus\{0\}}\frac{\varphi_*(|\delta|^q\nu)(\Delta(\zeta,R))}{\omega(S(\z))^{\frac{q}{p}}}
	\int_{D(0,\eta)}|f_n(\zeta)|^q\frac{\omega(S(\z))^{\frac{q}{p}}}{(1-|\z|)^2}\,dA(\zeta)\\
  &\quad+\sup_{\zeta\in\D\setminus D(0,\eta)}\frac{\varphi_*(|\delta|^q\nu)(\Delta(\zeta,R))}{\omega(S(\z))^{\frac{q}{p}}}
	\int_{\D\setminus D(0,\eta)}|f_n(\zeta)|^q\frac{\omega(S(\z))^{\frac{q}{p}}}{(1-|\z|)^2}\,dA(\zeta)\\
  &\lesssim\e\int_{\D}\frac{\omega(S(\z))^{\frac{q}{p}}}{(1-|\z|)^2}\,dA(\zeta)
	+\e\|f_n\|_{A^p_\om}^q\lesssim\e,\quad n\ge N.
  \end{split}
	\end{equation*}
Thus also the second quantity in \eqref{Eq:4-conditions-comp} tends to zero as $n\to\infty$ in the case $q>p$.

Finally, let $q=p$. The compactness of the identity operator from $A^p_\om$ to $L^p_{\varphi_*(|\delta|^p\nu)}$ implies $\varphi_*(|\delta |^p\nu)(S(\z))/\omega(S(\z))\to0$ as $|\z|\to1^-$ by \cite[Theorem~3(ii)]{PRS}. By following the proof above the only different step consists of making the quantity
	$$
	J(\eta)=\int_{\D\setminus D(0,\eta)}|f_n(\zeta)|^p\frac{\varphi_*(|\delta|^p\nu)(\Delta(\z,R))}{(1-|\z|)^2}\,dA(\z)
	$$
small, uniformly in $n$, by choosing $\eta\in(0,1)$ sufficiently large. But an application of \cite[Theorem~1(b)]{PelRatEmb} together with \eqref{pili} yields
	$$
	J(\eta)\lesssim\|f_n\|_{A^p_\om}^p\sup_{b\in\D\setminus\{0\}}\frac{\varphi_*(|\delta|^p\nu)(S(b)\setminus D(0,\eta))}{\om(S(b))}.
	$$
Standard arguments can now be used to make the right hand side smaller than a pregiven $\e>0$ for $\eta$ sufficiently large by using $\varphi_*(|\delta |^p\nu)(S(\z))/\omega(S(\z))\to0$ as $|\z|\to1^-$, see, for example, \cite[pp.~26--27]{PR} for details. This completes the proof of the proposition.
\end{proof}

The next result is a counter part of Proposition~\ref{proposition pq11} in the case $p>q$.

\begin{proposition}\label{proposition pq1}
Let $0<q<p<\infty$ and $\om\in\DD$, and let $\nu$ be a positive Borel measure on $\D$. Let $\f$ and $\p$ be analytic self-maps of $\D$. If $\d C_\f$ and $\d C_\p$ are bounded from $A_\om^p$ to $L^q_\nu$, then $C_\f-C_\p:A^p_\om\rightarrow L^q_\nu$ is compact.
\end{proposition}

\begin{proof}
Let $\{f_n\}$ be a bounded sequence in $A_\om^p$ such that $f_n\rightarrow0$ uniformly on compact subsets of $\D$. Since $\delta C_\f$ and $\delta C_\p$ are bounded from $A^p_\om$ to $L^q_{\nu}$ by the hypothesis, they are also compact by \cite[Theorem~3(iii)]{PRS}, and therefore
	\begin{equation}\label{eq51}
	\lim_{n\rightarrow\infty}\left(\|\delta f_n(\f)\|_{L^q_{\nu}}+\|\delta f_n(\p)\|_{L^q_{\nu}}\right)=0.
	\end{equation}
Let $0<r<R<1$, and denote $E=\{z\in\D:|\delta(z)|<r\}$ and $E'=\D\setminus E$. To prove the compactness of $C_{\varphi}-C_{\p}:A^p_\om\rightarrow L^q_\nu$, it suffices to show that
	$$
	\lim_{n\rightarrow\infty}(\|(C_\f-C_\p)(f_n)\chi_E\|_{L^q_{\nu}}+\|(C_\f-C_\p)(f_n)\chi _{E'}\|_{L^q_{\nu}})=0
	$$
since
	\begin{equation*}
	\|(C_\f-C_\p)(f_n)\|^q_{L^q_{\nu}}=\|(C_\f-C_\p)(f_n)\chi_E\|^q_{L^q_{\nu}}+\|(C_\f-C_\p)(f_n)\chi _{E'}\|^q_{L^q_{\nu}}.
	\end{equation*}
By using \eqref{Eq:easy} and \eqref{eq51}, it is easy to show that
	$$
	\lim_{n\rightarrow\infty}\|(C_\f-C_\p)(f_n)\chi _{E'}\|_{L^q_{\nu}}=0.
	$$
Further, by \eqref{Eq:estimate-sufficiency}, we have 	
	\begin{equation*}
	\|(C_\f-C_\p)(f_n)\chi_E\|^q_{L^q_{\nu}}\lesssim\int_\D|f_n(\z)|^q\frac{\varphi_*(|\delta |^q\nu)(\Delta(\z,R))}{(1-|\z|)^2}\,dA(\z).
	\end{equation*}
Let $\varepsilon>0$. Since the identity operator from $A^p_\om$ to $L^q_{\f_*(|\delta|^q\nu)}$ is bounded by the hypothesis, \cite[Theorem 3(iii)]{PRS} and the dominated convergence theorem imply the existence of an $R_0=R_0(\e)\in(0,1)$ such that
	\begin{equation}\label{eq52}
	\int_\D\left(\int_{\Gamma(z)\backslash\overline{D(0,R_0)}}\frac{\varphi_*(|\delta|^q\nu)(\Delta(\z,R))}
	{\om(T(\zeta))(1-|\z|)^2}\,dA(\z)\right)^\frac{p}{p-q}\om(z)\,dA(z)<\varepsilon^{\frac{p}{p-q}}.
	\end{equation}
Further, by the uniform convergence, there exists $N=N(\e)\in\mathbb{N}$ such that $|f_n(z)|<\varepsilon^{\frac1q}$ for all $n\ge N$ and $z\in\overline{D(0,R_0)}$. Therefore, for all $n\ge N$, by Fubini's theorem, H\"{o}lder's inequality, \cite[Lemma 4.4]{PR} and \eqref{eq52}, we have
	\begin{equation*}
	\begin{split}
	\|(C_\f-C_\p)(f_n)\chi_{E}\|^q_{L^q_{\nu}}&\lesssim \left(\int_{\overline{D(0,R_0)}}+\int_{\D\backslash\overline{D(0,R_0)}}\right)
	|f_n(\z)|^q\frac{\varphi_*(|\delta |^q\nu)(\Delta(\z,R))}{(1-|\z|)^2}\,dA(\z)\\
	&\lesssim \varepsilon+\int_\D\left(\int_{\Gamma(z)\backslash\overline{D(0,R_0)}}|f_n(\z)|^q\frac{\varphi_*(|\delta|^q\nu)(\Delta(\z,R))}
	{\om(T(\zeta))(1-|\z|)^2}\,dA(\z)\right)\om(z)\,dA(z)\\
	&\leq\varepsilon+\int_\D N(f_n)^q(z)\left(\int_{\Gamma(z)\backslash\overline{D(0,R_0)}}\frac{\varphi_*(|\delta|^q\nu)(\Delta(\z,R))}
	{\om(T(\zeta))(1-|\z|)^2}\,dA(\z)\right)\om(z)\,dA(z)\\
	&\leq\varepsilon\left(1+\|N(f_n)\|^q_{L^p_{\om}}\right)\asymp\varepsilon.
	\end{split}
	\end{equation*}
Therefore	$\lim_{n\rightarrow\infty}\|(C_\f-C_\p)(f_n)\chi_{E}\|_{L^q_{\nu}}=0$, and thus $C_\f-C_\p$ is compact from $A^p_\om$ to $L^q_\nu$.
\end{proof}

\section{Necessary conditions}\label{Sec:necessity}

In this section we establish necessary conditions for $C_\f-C_\p:A^p_\om\rightarrow L^q_\nu$ to be bounded or compact. In the case $p<q$ we work with the whole class $\DD$, but the arguments employed in the case $p\ge q$ rely strongly on the hypothesis $\om\in\DDD$. We begin with the case $p\le q$.

\begin{proposition} \label{proposition pq0}
Let either $0<p<q<\infty$ and $\om\in\DD$ or $p=q$ and $\om\in\DDD$, and let $\nu$ be a positive Borel measure on $\D$. Let $\f$ and $\p$ analytic self-maps of $\D$. If $C_\f-C_\p:A^p_\om\rightarrow L^q_\nu$ is bounded (resp. compact), then $\d C_\f$ and $\d C_\p$ are bounded (resp. compact) from $A_\om^p$ to $L^q_\nu$.
\end{proposition}

\begin{proof}
Let first $p<q$ and $\om\in\DD$. We begin with the boundedness and show in detail that $\d C_\f$ is bounded from $A_\om^p$ to $L^q_\nu$. For each $a\in\D$, consider the function
	\begin{equation*}
  f_{a}(z)=\left(\frac{1-|a|}{1-\overline{a}z}\right)^{\gamma}\omega(S(a))^{-\frac{1}{p}},\quad z\in\D,
  \end{equation*}
induced by $\om$ and $0<\gamma,p<\infty$. Then \cite[Lemma~2.1]{PelSum14} implies that for each $\gamma=\gamma(\om,p)>0$ sufficiently large we have $\|f_{a}\|_{A^p_\om}\asymp1$ for all $a\in\D$. Fix such a $\gamma$. Since $C_\varphi-C_\psi$ is bounded, we have
  \begin{equation*}
	\begin{split}
  1&\asymp \|f_{a}\|^q_{A^p_\omega}
  \gtrsim\|(C_\varphi-C_\psi)(f_{a})\|^q_{L^q_\nu}\\
  &=\int_\D\left|\left(\frac{1-|a|}{1-\overline{a}\varphi(z)}\right)^{\gamma}-\left(\frac{1-|a|}{1-\overline{a}\psi(z)}\right)^{\gamma}\right|^q\frac{\nu(z)}{\omega(S(a))^{\frac{q}{p}}}\,dA(z)\\
  &=\int_\D \left|\frac{1-|a|}{1-\overline{a}\varphi(z)}\right|^{\gamma q}
	\left|1-\left(\frac{1-\overline{a}\varphi(z)}{1-\overline{a}\psi(z)}\right)^{\gamma}\right|^q
	\frac{\nu(z)}{\omega(S(a))^{\frac{q}{p}}}\,dA(z).
  \end{split}
  \end{equation*}
According to \cite[p.~795]{SE}, for each $0<\gamma<\infty$ and $0<r<1$ there exist a constant $C=C(\gamma,r)>0$ such that
  \begin{equation}\label{Eq:Old-Propositions}
  \left|1-\left(\frac{1-\overline{a}z}{1-\overline{a}w}\right)^\gamma\right|
	\ge C|a|\rho(z,w),\quad z\in\De(a,r),\quad a,w\in\D.
  \end{equation}
An application of this inequality gives
	\begin{equation*}
  1\gtrsim\int_{\varphi^{-1}(\De(a,r))}\frac{|a|^q|\delta(z)|^q}{(\omega(S(a)))^{\frac{q}{p}}}\nu(z)\,dA(z)
	=|a|^q\frac{\varphi_*(|\delta|^q\nu)(\De(a,r))}{(\omega(S(a)))^{\frac{q}{p}}}.
  \end{equation*}
It follows that $\varphi_*(|\delta|^q\nu)$ is a bounded $q$-Carleson measure for $A^p_\om$ by \cite[Theorem~1(c)]{PelRatEmb}, and hence $\d C_\f:A_\om^p\to L^q_\nu$ is bounded. The same argument shows that also $\d C_\p$ is bounded.

For the compactness statement, first observe that $f_{a}$ tends to zero uniformly on compact subsets of $\D$ as $|a|\rightarrow 1^-$. Then, if $C_\varphi-C_\psi$ is compact, we have $\lim_{|a|\rightarrow 1^-}\|(C_\varphi-C_\psi)(f_{a})\|_{L^q_\nu}=0$. By arguing as above we deduce
	$$
	\lim_{|a|\rightarrow 1^-} \frac{\varphi_*(|\delta|^q\nu)(\De(a,r))}{(\omega(S(a)))^{\frac{q}{p}}}=0.
  $$
Therefore $\delta C_\varphi:A^p_\omega\to L^q_\nu$ is compact by \cite[Theorem~3]{PRS}. The same argument shows that also $\delta C_\psi$ is compact.

Let now $p=q$ and $\om\in\DDD$. The statement follows from the proof above with the modification that \cite[Theorem~2]{LR2020}, valid for $\om\in\DDD$, is used instead of \cite[Theorem~1(c)]{PelRatEmb} and \cite[Theorem~3]{PRS}. The only extra step is to observe that for each $\om\in\DDD$ there exists $r=r(\om)\in(0,1)$ such that $\om(S(a))\asymp\om(\Delta(a,r))$ for all $a\in\D\setminus\{0\}$. This follows from \eqref{Eq:characterization-D}. With this guidance we consider the proposition proved.
\end{proof}

The next result establishes a counter part of Proposition~\ref{proposition pq0} when $p>q$.

\begin{proposition}\label{proposition:necessity-p>q}
Let $0<q<p<\infty$ and $\om\in\mathcal{D}$, and let $\nu$ be a positive Borel measure on $\D$. Let $\f$ and $\p$ be analytic self-maps of $\D$. If $C_\f-C_\p$ : $A^p_\om \rightarrow L^q_\nu$ is bounded, then $\d C_\f$ and $\d C_\p$ are both bounded from $A_\om^p$ to $L^q_\nu$.
\end{proposition}

\begin{proof}
Let $\{z_k\}_{k\in\N}$ be a $\r$-lattice such that it is ordered by increasing modulii and $z_k\neq0$ for all $k$. Then by \cite[Theorem~1]{JJK2} there exist constants $M=M(p,\om)>1$ and $C=C(p,\om)>0$ such that the function
	$$
	F(z)=\sum_kb_k\left(\frac{1-|z_k|}{1-\overline{z_k}z}\right)^M\frac{1}{(\omega(T(z_k)))^{\frac{1}{p}}},\quad z\in\D,
	$$
belongs to $A^p_{\omega}$ and satisfies $\|F\|_{A^p_\om}\le C\|b\|_{\ell^p}$ for all $b=\{b_k\}\in\ell^p$. Since $C_\f-C_\p$ : $A^p_\om \rightarrow L^q_\nu$ is bounded by the hypothesis, we deduce
	\begin{equation*}
	\begin{split}
	\|b\|_{\ell^p}^q
	&\gtrsim\|F\|_{A^p_{\om}}^q
	\gtrsim\int_{\D}|(C_{\varphi}-C_{\psi})\circ F(z)|^q\,d\nu(z)\\
	&=\int_{\D}\left|\sum_kb_k\left(\left(\frac{1-|z_k|}{1-\overline{z_k}\varphi(z)}\right)^{M}-
	\left(\frac{1-|z_k|}{1-\overline{z_k}\psi(z)}\right)^{M}\right)\frac{1}{(\omega(T(z_k)))^{\frac1p}}\right|^qd\nu(z),\quad b\in\ell^p.
	\end{split}
	\end{equation*}
We now replace $b_k$ by $b_kr_k(t)$, integrate with respect to $t$ from 0
to 1, and then apply Fubini's theorem and Khinchine's inequality to get
	$$
	\|b\|_{\ell^p}^q
	\gtrsim\int_{\D}\left(\sum_k|b_k|^2\left|\left(\frac{1-|z_k|}{1-\overline{z_k}\varphi(z)}\right)^{M}-
	\left(\frac{1-|z_k|}{1-\overline{z_k}\psi(z)}\right)^{M}\right|^2 \frac{1}{(\omega(T(z_k)))^{\frac2p}}\right)^{\frac q2}d\nu(z),\quad b\in\ell^p.
	$$
By applying \eqref{Eq:Old-Propositions} and the estimate $|1-\overline{z_k}z|\asymp1-|z_k|$, valid for all $z\in\Delta(z_k,\r)$ and $k\in\N$, we obtain
	\begin{align*}
	\left|\left(\frac{1-|z_k|}{1-\overline{z_k}\varphi(z)}\right)^{M}-
	\left(\frac{1-|z_k|}{1-\overline{z_k}\psi(z)}\right)^{M}\right|
	&=\left|\frac{1-|z_k|}{1-\overline{z_k}\varphi(z)}\right|^{M}
	\left|1-\left(\frac{1-\overline{z_k}\varphi(z)}{1-\overline{z_k}\psi(z)}\right)^{M}\right|\\
	&\gtrsim|z_k||\delta(z)|\left|\frac{1-|z_k|}{1-\overline{z_k}\varphi(z)}\right|^{M}\chi_{\varphi^{-1}(\Delta(z_k,\r))}(z)\\
	&\asymp|z_k||\delta(z)|\chi_{\varphi^{-1}(\Delta(z_k,\r))}(z),\quad z\in\D,\quad k\in\N,
	\end{align*}
and hence
	\begin{equation}\label{Eq:1111}
	\|b\|_{\ell^p}^q
	\gtrsim|z_1|^q\int_{\D}\left(\sum_{k}|b_k|^2|\delta(z)|^2\frac{1}{(\omega(T(z_k)))^{\frac2p}}\chi_{\varphi^{-1}(\Delta(z_k,\r))}(z)\right)^{\frac{q}{2}}d\nu(z).
	\end{equation}
If $q\ge2$ then the inequality $\sum_j{c_j^x}\le\left(\sum_jc_j\right)^x$, valid for all $c_j\ge0$ and $x\ge1$, imply
	\begin{equation}\label{2222}
	\begin{split}
	\|b\|_{\ell^p}^q
	&\gtrsim \int_{\D}\left(\sum_{k}|b_k|^q|\delta(z)|^q\frac{1}{(\omega(T(z_k)))^{\frac{q}p}}\chi_{\varphi^{-1}(\Delta(z_k,\r))}(z)\right)d\nu(z).
	\end{split}
	\end{equation}
To get the same estimate for $0<q<2$ we apply H\"{o}lder's inequality. It together with the fact that the number of discs $\Delta(z_k,r)$ to which each $\vp(z)$ may belong to is uniformly bounded yields
	\begin{equation*}
	\begin{split}
	&\int_{\D}\left(\sum_{k}|b_k|^q|\delta(z)|^q\frac{1}{(\omega(T(z_k)))^{\frac{q}{p}}}
	\chi_{\varphi^{-1}(\Delta(z_k,\r))}(z)\right)d\nu(z)\\
	&\quad\le\int_{\D}\left(\sum_{k}|b_k|^2|\delta(z)|^2\frac{1}{(\omega(T(z_k)))^{\frac{2}{p}}}
	\chi_{\varphi^{-1}(\Delta(z_k,r))}(z)\right)^{\frac{q}{2}}\\
	&\qquad\cdot\left(\sum_{k}\chi_{\varphi^{-1}(\Delta(z_k,\r))}(z)\right)^{1-\frac{q}{2}}d\nu(z)\\
	&\quad\lesssim\int_{\D}\left(\sum_{k}|b_k|^2|\delta(z)|^2\frac{1}{(\omega(T(z_k)))^{\frac{2}{p}}}
	\chi_{\varphi^{-1}(\Delta(z_k,\r))}(z)\right)^{\frac{q}{2}}d\nu(z).
	\end{split}
	\end{equation*}
Thus \eqref{2222} holds for each $0<q<\infty$. By using Fubini's theorem we now deduce	
	\begin{equation*}
	\begin{split}
	\|b\|_{\ell^p}^q
	&\gtrsim\sum_{k}|b_k|^q\frac{1}{(\omega(T(z_k)))^{\frac{q}{p}}}\int_{\varphi^{-1}(\Delta(z_k,\r))}|\delta(z)|^q d\nu(z)\\
	&=\sum_{k}|b_k|^q\frac{\varphi_{*}(|\delta|^q \nu)(\Delta(z_k,\r))}{{(\omega(T(z_k)))^{\frac{q}{p}}}},\quad b\in\ell^p.
	\end{split}
	\end{equation*}
Therefore the sequence
	$$
	\left\{\frac{\varphi_{*}(|\delta|^q \nu)(\Delta(z_k,\r))}{{(\omega(T(z_k)))^{\frac{q}{p}}}}\right\}_{k\in\N}
	$$
belongs to $(\ell^{\frac{p}{q}})^*\simeq\ell^{\frac{p}{p-q}}$, and consequently
 $$
 \sum_k\left(\frac{\varphi_{*}(|\delta|^q\nu)(\Delta(z_k,\r))}{(\omega(T(z_k)))^{\frac{q}{p}}}\right)^{\frac{p}{p-q}}<\infty.
 $$
Pick up an $r=r(\rho)\in(0,1)$ such that $\Delta(z,\rho)\subset\Delta(z_k,r)$ for all $z\in\Delta(z_k,\rho)$ and $k\in\N$. The right hand inequality in \eqref{Eq:characterization-D} shows that $\widehat{\om}(z)\asymp\widehat{\om}(z_k)$ and $\omega(S(z))\asymp\omega(S(z_k))$ for all $z\in\Delta(z_k,\r)$ and $k\in\N$. Then, as $\{z_k\}_{k\in\N}$ is a $\r$-lattice, we deduce
	\begin{equation*}
	\begin{split}
	&\int_{\D}\left(\frac{\varphi_{*}(|\delta|^q \nu)(\Delta(z,\rho))}{\omega(S(z))}\right)^{\frac{p}{p-q}}\widetilde{\omega}(z)dA(z)\\
	&\quad\asymp\sum_k\int_{\Delta(z_k,\rho)}\left(\frac{\varphi_{*}(|\delta|^q \nu)(\Delta(z,\rho))}{\omega(S(z))}\right)^{\frac{p}{p-q}}
	\widetilde{\omega}(z)dA(z)\\
	&\quad\lesssim\sum_k\left(\frac{\varphi_{*}(|\delta|^q\nu)(\Delta(z_k,r))}{\omega(S(z_k))}\right)^{\frac{p}{p-q}}\widetilde{\omega}(z_k)(1-|z_k|)^2\\
	&\quad\lesssim\sum_k\left(\frac{\varphi_{*}(|\delta|^q\nu)(\Delta(z_k,r))}{(\omega(S(z_k)))^\frac{q}{p}}\right)^{\frac{p}{p-q}}
	\lesssim\sum_k\left(\frac{\varphi_{*}(|\delta|^q\nu)(\Delta(z_k,\r))}{(\omega(T(z_k)))^\frac{q}{p}}\right)^{\frac{p}{p-q}}<\infty.
	\end{split}
	\end{equation*}
Therefore $\varphi_{*}(|\delta|^q \nu)$ is a $q$-Carleson measure for $A^p_{\omega}$ by Theorem~\ref{Theorem:cm}. For the same reason, $\psi_{*}(|\delta|^q \nu)$ is a $q$-Carleson measure for $A^p_{\omega}$. The proof is complete.
\end{proof}

\section{Proofs of main theorems}\label{final}

Here we shortly indicate how the main results stated in the introduction easily follow from the propositions proved in the previous two sections.

\medskip
\noindent\emph{Proof of} Theorem~\ref{Theorem:main-1}.
The theorem follows by Propositions~\ref{proposition pq1} and \ref{proposition:necessity-p>q}. Namely, if $\delta C_\varphi$ and $\delta C_\varphi$ are bounded from $A^p_\omega$ to $L^q_\nu$, then $C_\vp-C_\psi:A^p_\om\to L^q_\nu$ is compact, and thus bounded as well, by Proposition~\ref{proposition pq1}. Conversely, if $C_\vp-C_\psi:A^p_\om\to L^q_\nu$ is bounded, then $\delta C_\varphi$ and $\delta C_\varphi$ are bounded by Proposition~\ref{proposition:necessity-p>q}.
\hfill$\Box$

\medskip
\noindent\emph{Proof of} Theorem~\ref{Theorem:main-2}.
The theorem is an immediate consequence of Propositions~\ref{proposition pq11} and \ref{proposition pq0}.
\hfill$\Box$


\begin{thebibliography}{99}


\bibitem{AW}            Acharyya, Soumyadip; Wu, Zhijian: Compact and
                        Hilbert-Schmidt differences of weighted composition
                        operators. Integral Equations Operator Theory 88
                        (2017), 465--482.

\bibitem{CB}            Choe, Boo Rim; Choi, Koeun; Koo, Hyungwoon; Yang,
                        Jongho: Difference of weighted composition operators. J.
                        Funct. Anal. 278 (2020), 108401, 38 pp.


\bibitem{CKP}           Choe, Boo Rim; Koo, Hyungwoon; Park, Inyoung:
                        Compact differences of composition operators on the Bergman spaces over the ball. Potential Anal. 40 (2014), 81¨C102.


\bibitem{CCM}           Cowen, Carl C.; MacCluer, Barbara D.: Composition
                        operators on spaces of analytic functions. Studies in
                        Advanced Mathematics. CRC Press, Boca Raton, FL,
                        1995. xii+388 pp.

\bibitem{DS}           Duren, Peter; Schuster, Alexander Bergman spaces. Mathematical Surveys and
                        Monographs,100. American Mathematical Society, Providence, RI, 2004.

\bibitem{GTE}           Goebeler, Thomas E., Jr.: Composition operators acting
                        between Hardy spaces. Integral Equations Operator Theory
                        41 (2001), 389--395.
                        
\bibitem{PM}            Halmos, Paul R. : Measure Theory, Springer-Verlag, New York, 1974.

\bibitem{HKZ}           Hedenmalm, Haakan; Korenblum, Boris; Zhu, Kehe Theory of Bergman spaces. Graduate
                        Texts in Mathematics, 199. Springer-Verlag, New York, 2000.

\bibitem{KW}            H. Koo and M. Wang, {\it Joint Carleson measure and the difference
                        of composition operators on $A^p_{\alpha}(B_n)$,}
                        {\sl J. Math. Anal. Appl.}, {\bf419}(2014), 1119-1142.

\bibitem{LR2020}		Liu, Bin; R\"atty\"a, Jouni,
					    Compact differences of weighted composition operators,
						preprint. https://arxiv.org/pdf/2005.12016.pdf

\bibitem{L1993}			Luecking, Daniel H., Embedding theorems for spaces of analytic functions
                         via Khinchine's inequality, Michigan Math. J. 40 (1993), 333--358.

\bibitem{MJ}            Moorhouse, Jennifer: Compact differences of
                        composition operators. J. Funct. Anal. 219 (2005),
                       70--92.

\bibitem{MT}            Moorhouse, Jennifer; Toews, Carl: Differences of
                        composition operators. Trends in Banach spaces and
                        operator theory (Memphis, TN, 2001), 207--213,
                        Contemp. Math., 321, Amer. Math. Soc., Providence,
                        RI, 2003.

\bibitem{NS}                Nieminen, Pekka J.; Saksman, Eero: On compactness
                            of the difference of composition operators.
                            J. Math. Anal. Appl. 298 (2004), 501--522.
                            
\bibitem{PelSum14}          Pel\'aez, Jos\'e \'Angel: Small weighted Bergman
                            spaces, Proceedings of the summer school in
                            complex and harmonic analysis, and related
                            topics, (2016).    
                            
\bibitem{PR}            Pel\'aez, Jos\'e \'Angel; R\"atty\"a, Jouni: Weighted
                        Bergman spaces induced by rapidly increasing weights.
                        Mem. Amer. Math. Soc. 227 (2014), no. 1066, vi+124pp.

\bibitem{PelRat2020}    Pel\'aez, Jos\'e \'Angel and R\"atty\"a, Jouni:
                        Bergman projection induced by radial weight, https://
                        arxiv.org/pdf/1902.09837.pdf.

\bibitem{PelRatEmb}     Pel\'aez, Jos\'e \'Angel and R\"atty\"a, Jouni:
                        Embedding theorems for Bergman spaces via harmonic
                        analysis, Math. Ann. 362 (2015), 205--239.

\bibitem{PRS}           Pel\'aez, Jos\'e \'Angel; R\"atty\"a, Jouni;
                        Sierra, Kian: Embedding Bergman spaces into tent spaces,
												Math.Z, 281(2015),1215--1237.

\bibitem{PelRatSchatten}    Pel\'aez, Jos\'e \'Angel; R\"atty\"a, Jouni:
                            Trace class criteria for Toeplitz and composition
                            operators on small Bergman spaces, Adv.~Math. 293
                            (2016), 606--643.
                            
\bibitem{JJK2}              Pel\'aez, Jos\'e \'Angel; R\"atty\"a, Jouni;
                            Sierra, Kian: Atomic decomposition and Carleson
                            measures for weighted Mixed norm spaces. J. Geom. Anal. (2019), 1-33.

\bibitem{SE}            Saukko, Erno: Difference of composition operators
                        between standard weighted Bergman spaces. J. Math.
                        Anal. Appl. 381 (2011), 789--798.

\bibitem{SEO}           Saukko, Erno: An application of atomic decomposition
                        in Bergman spaces to the study of differences of
                        composition operators. J. Funct. Anal. 262 (2012),
                        3872--3890.

\bibitem{SJH1}          Shapiro, Joel H.: The essential norm of a composition
                        operator. Ann. of Math. (2) 125 (1987), 375--404.

\bibitem{SJH2}          Shapiro, Joel H.: Composition operators and classical
                        function theory. Universitext: Tracts in Mathematics.
                        Springer-Verlag, New York, 1993.

\bibitem{SS}            Shapiro, Joel H.; Sundberg, Carl: Isolation amongst
                        the composition operators. Pacific J. Math. 145
                        (1990), 117--152.

\bibitem{SL}            Shi, Yecheng; Li, Songxiao Difference of composition operators between 
                        different Hardy spaces. J. Math. Anal. Appl. 467 (2018), 1--14.

\bibitem{SLD}           Shi,Y.; Li,S.; Du,J.: Difference of composition operators between weighted
                        Bergman spaces on the unit ball. https://arxiv.org/pdf/1903.00651.pdf.

\bibitem{Z}             Zhu, K.: Operator Theory in Function Spaces.
                        American Mathematical Society, Providence, RI. 2007.
\end{thebibliography}
\end{document}